\documentclass[english,11pt,reqno]{smfart}

\usepackage{color}
\usepackage{amsthm}
\usepackage{amsmath}
\usepackage{amsfonts}
\usepackage{amstext}
\usepackage[latin1]{inputenc}
\usepackage{amscd}
\usepackage{latexsym}
\usepackage{bm}
\usepackage{amssymb}
\usepackage[all]{xy}
\usepackage{euscript}
\usepackage{a4wide}
\usepackage{amsmath,amssymb,graphicx}
\usepackage{amssymb}
\usepackage{amsmath}
\usepackage{mathrsfs,mathtools}

\newtheorem{theorem}{Theorem}
\newtheorem{proposition}[theorem]{Proposition}
\newtheorem{lemma}[theorem]{Lemma}

\newtheorem{example}{Example}
\newtheorem{definition}[theorem]{Definition}

\def\di{\displaystyle}

\newcommand{\R}{\mathbb{R}}
\newcommand{\TT}{\mathbb{T}}
\newcommand{\RN}{\mathbb{R}^{N+1}}
\newcommand{\Rn}{\mathbb{R}^{N-1}}

\newcommand{\LL}{\mathscr{L}}
\newcommand{\CC}{\mathscr{C}}
\newcommand{\Q}{\boldsymbol{Q}}
\newcommand{\ZZ}{\boldsymbol{Z}}
\newcommand{\T}{\boldsymbol{T}}
\newcommand{\W}{\boldsymbol{W}}
\newcommand{\PP}{\overline{P}}
\newcommand{\OO}{\overline{O}}

\newcommand{\fonction}[5]{\begin{array}[t]{lrcl}#1 :&#2 &\longrightarrow &#3\\&#4& \longmapsto &#5 \end{array}}

\newcommand{\fonctionsansens}[3]{\begin{array}[t]{lrcl}#1 :&#2 &\longmapsto &#3 \end{array}}

\begin{document}
\title{Helmholtz's inverse problem of the discrete calculus of variations}
\author{Lo\"ic Bourdin and Jacky Cresson}
\maketitle

\begin{abstract}
We derive the discrete version of the classical Helmholtz's condition. Precisely, we state a theorem characterizing second order finite differences equations admitting a Lagrangian formulation. Moreover, in the affirmative case, we provide the class of all possible Lagrangian formulations.
\end{abstract}

\tableofcontents

\section{Introduction}\label{section1}
A classical problem in Analysis is the well-known {\it Helmholtz's inverse problem of the calculus of variations} (see \cite[p.71]{doug}, \cite{helm} and \cite[p.377]{olve}): find a necessary and sufficient condition under which a (system of) differential equation(s) can be written as an Euler-Lagrange equation and, in the affirmative case, find all the possible Lagrangian formulations. \\

This problem has been studied by numerous authors and has been completely solved by A. Mayer \cite{maye} and A. Hirsch \cite{hirs,hirs2}. The formulation that we use is due to V. Volterra \cite{volt}. Precisely, let $O$ be a second order differential operator. Then, the differential equation $O(q)=0$ can be written as a second order Euler-Lagrange equation if and only if all the Frechet derivatives of $O$ are self-adjoint. This condition is usually called \textit{Helmholtz's condition}. We refer to \cite{olve} for a modern presentation and a complete proof of this theorem. \\

A more difficult problem is to deduce from this characterization a complete classification of second order differential equations admitting a variational formulation. This has been only solved in dimension three by J. Douglas in his seminal paper \cite{doug} following a previous work of D.R. Davis \cite{davi,davi2}. We refer to \cite[p.74-75]{doug} and \cite[p.377-379]{olve} for a historical survey.\\

In recent years, an increasing activity has been devoted to a discrete version of the calculus of variations in the context of the geometric numerical integration. We refer to the book \cite{lubi} and the review paper \cite{mars} for more details (see also \cite{cmt}). In this context, a {\it second order discrete Euler-Lagrange equation} is given by: 
\begin{multline}
\di \frac{\partial L_-}{\partial x} ( \Q,\Delta_- \Q,\T,h )+\di \frac{\partial L_+}{\partial x} ( \Q,-\Delta_+ \Q,\T,h ) \\ + \Delta_+ \left( \frac{\partial L_-}{\partial v} ( \Q,\Delta_- \Q,\T,h) \right) - \Delta_- \left( \frac{\partial L_+}{\partial v} ( \Q,-\Delta_+ \Q,\T,h) \right)  = 0, 
\end{multline}
for a given couple of Lagrangian $(L_-,L_+)$ and where $\T$ is a bounded regular partition of $\R$ associated to the step size of discretization $h$. $\Delta_-$ (resp. $\Delta_+$) is the backward (resp. forward) finite differences operator associated to $\T$.\\

In this framework, we formulate the {\it Helmholt'z inverse problem of the discrete calculus of variations} as follows: find a necessary and sufficient condition under which a second order finite differences equation (see Definition \ref{defSOFDE}) can be written as a second order discrete Euler-Lagrange equation and, in the affirmative case, find all the possible Lagrangian formulations. \\

This problem has been studied by numerous authors. We refer in particular to the work of Albu-Opris \cite{albu} and Cracium-Opris \cite{opri} as well as Hydon-Mansfeld \cite{hydon}. However, in each of these papers, the discrete Helmholtz's problem characterizes finite differences equations which correspond to Euler-Lagrange equations of a discrete Lagrangian functional that has not always a continuuum limit (see \cite[$\S$.5.3, p.213]{hydon}), \textit{i.e.} which does not correspond to a discretization of a continuous Lagrangian functional which is our main concern in this paper. As a consequence, these papers can not be used to provide an answer to the Helmholtz's problem in the context of geometric numerical integration. A discussion of this problem can be found in \cite[$\S$.5.3, p.213-214]{hydon}. 

\section{Second order finite differences equations}
\label{section2}

\subsection{Partitions and discrete derivative operators}
In the whole paper, let us consider the following set:
\begin{equation}
\TT := \{ \T = (t_p)_{p=0,...,N} \in \RN  \; \text{with} \; N \geq 4 \; \text{and} \; \exists h >0, \; \forall i=0,...,N-1, \; t_{i+1}-t_i = h \}.
\end{equation}
$\TT$ is a set of bounded regular partitions $\T$ of $\R$. Hence, for any partition $\T \in \TT$, an integer $N = \text{card} (\T) \geq 4$ and a step size of discretization $h > 0$ are associated. Consequently, for any $\T \in \TT$, we can also associate the following \textit{discrete derivative operators}:
\begin{equation}
\fonction{\Delta_-}{\R^{N+1}}{\R^{N}}{\Q}{ \left ( \dfrac{Q_{p} -Q_{p-1}}{h}  \right )_{p=1,\dots ,N}}
\end{equation}
and
\begin{equation}
\fonction{\Delta_+}{\R^{N+1}}{\R^{N}}{\Q}{\left ( \dfrac{Q_{p} -Q_{p+1}}{h}  \right )_{p=0,\dots ,N-1}.}
\end{equation}
Let us note that $\Delta_-$ (resp. $-\Delta_+$) is the classical backward (resp. forward) Euler approximation of the derivative operator $d/dt$. Moreover, $\Delta_-$ and $-\Delta_+$ commute and the discrete operator $-\Delta_+ \circ \Delta_- $ corresponds to the classical centered approximation of $d^2/dt^2$. \\

All these previous discrete elements depend on $\T \in \TT$. For simplicity, we omit this dependence in the notations. \\

In this paper, we are interested in the discretization of differential equations defined over real intervals $[a,b]$. For any reals $a<b$, we denote by $\TT_{a,b}$ the set of regular partitions $\T$ of the interval $[a,b]$ defined by :
\begin{equation}
\TT_{a,b} := \{ \T \in \TT \; \text{with} \; 0 \leq t_0 - a < h \; \text{and} \; 0 \leq b-t_N < h \}.
\end{equation}

\subsection{Properties of the discrete derivative operators} In this section, we first remind the classical discrete versions of the Leibniz formula and the integration by part. The proofs respectively follow from \cite[$\S$.2.51, p.34-35]{milne} and \cite[$\S$.2.64]{milne}.

\begin{lemma}[Discrete Leibniz formulas]\label{dlf}
Let $\T \in \TT$ and $\Q$, $\W \in \R^{N+1}$. It holds:
\begin{equation}
\forall p=1,...,N, \;\; (\Delta_- \Q\W)_{p} = (\Delta_- \Q)_{p} W_p + Q_{p-1} (\Delta_- \W)_{p} 
\end{equation}
and
\begin{equation}
\forall p=0,...,N-1, \; \;  (\Delta_+ \Q\W)_{p} = (\Delta_+ \Q)_{p} W_p + Q_{p+1} (\Delta_+ \W)_{p}.
\end{equation}
Finally, for any $p=1,...,N-1$, it holds:
\begin{multline}
(-\Delta_+ \circ \Delta_- \Q\W)_p = (-\Delta_+ \circ \Delta_- \Q)_p W_p + \Q_p (-\Delta_+ \circ \Delta_- \W)_p \\ + (-\Delta_+ \Q)_p (-\Delta_+ \W)_p + (\Delta_- \Q)_p (\Delta_- \W)_p.
\end{multline}
\end{lemma}

For any $\T \in \TT$, let us denote by $\R^{N+1}_{0,0} := \{ \W \in \RN, \; W_0 = W_1 = W_{N-1} = W_N = 0 \}$. 

\begin{lemma}[Discrete integrations by part]\label{dibp}
Let $\T \in \TT$ and $(\Q, \W) \in \R^{N+1} \times \R^{N+1}_{0,0}$. It holds:
\begin{equation}
\di \sum_{p=1}^{N-1} Q_p (\Delta_- \W)_p =  \di \sum_{p=2}^{N-2} (\Delta_+ \Q)_p W_p , \quad \text{and} \quad \di \sum_{p=1}^{N-1} Q_p (\Delta_+ \W)_p =  \di \sum_{p=2}^{N-2} (\Delta_- \Q)_p W_p.
\end{equation}
Finally, it holds:
\begin{equation}
\di \sum_{p=1}^{N-1} Q_p (-\Delta_+ \circ \Delta_- \W)_p =  \di \sum_{p=2}^{N-2} (-\Delta_+ \circ \Delta_- \Q)_p W_p.
\end{equation}
\end{lemma}

We refer to \cite[p.42]{milne} for a formal relation between the classical continuous integration by part and the previous discrete ones (called {\it summations by part in \cite{milne}}). As we always construct discrete analogous of continuous notions, we prefer to keep the terminology of \textit{discrete integration by part} instead of \textit{summation by part} which does not refer to a discrete quadrature formula for the approximation of integrals.

\subsection{Second order finite differences equations}
In the whole paper, we assume that all the elements denoted by $\OO$, $\PP$, $L$, $L_-$ and $L_+$ are sufficiently smooth in order to justify all the computations. \\ 

In the continuous case, a second order differential equation on an interval $[a,b]$ is defined by $O^{a,b}(q) = 0$ where $O$ is a second order differential operator, \textit{i.e.}:
\begin{equation}
\fonctionsansens{O}{a<b}{\fonction{O^{a,b}}{\CC^2([a,b],\R)}{\CC^0([a,b],\R)}{q}{ O^{a,b}(q) }}
\end{equation}
with:
\begin{equation}
\fonction{O^{a,b}(q)}{[a,b]}{\R}{t}{O^{a,b}(q)(t) = \OO \big( q(t),\dot{q}(t), \ddot{q}(t),t \big)}
\end{equation}
where $\dot{q}$ (resp. $\ddot{q}$) is the first (resp. second) derivative of $q$ and where:
\begin{equation}\label{pbarcont}
\fonction{\OO}{\R^4}{\R}{(x,v,w,t)}{\OO(x,v,w,t).}
\end{equation}
Hence, a second order differential equation (independently of the interval $[a,b]$) is entirely determined by the application $\OO$. A discrete analogue of this definition reads as follows :

\begin{definition}\label{defSOFDE}
A second order finite differences equation, associated to a partition $\T \in \TT$, is defined by $P^{\T} (\Q) = 0$ where $P$ is a second order finite differences operator, \textit{i.e.}:
\begin{equation}
\fonctionsansens{P}{\T \in \TT}{\fonction{P^{\T}}{\RN}{\Rn}{\Q}{ P^{\T}(\Q)  = \big( P^{\T}_p(\Q) \big)_{p=1,...,N-1} }}
\end{equation}
where
\begin{equation}
\forall p=1,...,N-1 , \; P^{\T}_p(\Q) = \PP \big( Q_p, (\Delta_- \Q)_p,(-\Delta_+ \Q)_p ,(-\Delta_+ \circ \Delta_- \Q)_p ,t_p,h \big)
\end{equation}
and where
\begin{equation}\label{pbardisc}
\fonction{\PP}{\R^5 \times \R^+_*}{\R}{(x,v_-,v_+,w,t,\xi)}{\PP(x,v_-,v_+,w,t,\xi).}
\end{equation}
A second order finite differences equation (independently of the partition $\T \in \TT$) is then entirely determined by the application $\PP$.
\end{definition}

Let us consider a second order differential equation $O^{a,b}(q) = 0$ on an interval $[a,b]$. A usual algebraic way in order to provide a discretization of this equation is to consider a partition $\T \in \TT_{a,b}$ and to define :
\begin{equation}
\fonction{\PP}{\R^5 \times \R^+_*}{\R}{(x,v_-,v_+,w,t,\xi)}{\OO \big( x,(1-\lambda) v_- +\lambda v_+,w,t \big)}
\end{equation}
with $\lambda \in [0,1]$. We then obtain the numerical scheme $P^{\T} (\Q) = 0$. The parameter $\lambda$ allows to choose for example the backward ($\lambda =0$), centered ($\lambda =1/2$) or forward ($\lambda =1$) approximation of the derivative $d/dt$. Such a discretization of $O^{a,b}(q) = 0$ is called a \textit{direct discretization}.

\begin{example}
Let us consider the Newton's equation with friction $q+\dot{q}+\ddot{q} = 0$ defined on a real interval $[a,b]$. It is a second order differential equation associated to $\OO (x,v,w,t) = x+v+w$. Hence, considering $\lambda =1/2$ and a partition $\T \in \TT_{a,b}$, we obtain by direct discretization the following numerical scheme:
\begin{equation}
\forall p=1,...,N-1, \; Q_p+ \dfrac{Q_{p+1}-Q_{p-1}}{2h}+\dfrac{Q_{p+1}-2Q_p + Q_{p-1}}{h^2}  = 0.
\end{equation}
\end{example}

\section{Formulation of the discrete version of the Helmholtz's problem for second order finite differences equations}\label{section3}
\subsection{Reminder about the classical Helmholtz's result for second order differential equations} A continuous Lagrangian system derives from a variational principle. Precisely, let us consider two reals $a<b$ and the following Lagrangian functional:
\begin{equation}
\fonction{\LL^{a,b}}{\CC^2([a,b],\R)}{\R}{q}{\di \int_a^b L(q,\dot{q},t) \; dt,}
\end{equation}
where $L$ is a Lagrangian, \textit{i.e.} an application of the type :
\begin{equation}
\fonction{L}{\R^3}{\R}{(x,v,t)}{L(x,v,t).}
\end{equation}
Let $\CC^2_0 ([a,b],\R) := \{ w \in \CC^2([a,b],\R), \; w(a) = w(b) = 0\}$ denote the set of \textit{variations}. Then, $q \in \CC^2([a,b],\R)$ is said to be a \textit{critical point} of $\LL^{a,b}$ if for any variation $w$, $D\LL^{a,b} (q)(w)=0$. The calculus of variations allows to characterize the critical points of $\LL^{a,b}$ as the solutions on $[a,b]$ of the following \textit{second order Euler-Lagrange equation}:
\begin{equation}\label{el}\tag{EL${}^{a,b}$}
\dfrac{\partial L}{\partial x} (q,\dot{q},t) - \dfrac{d}{dt} \left( \dfrac{\partial L}{\partial v} (q,\dot{q},t) \right) = 0.
\end{equation}
A dynamical system governed by such an Euler-Lagrange equation is called a \textit{second order Lagrangian system}. We refer to \cite[p.55-57]{arno} for more details concerning continuous Lagrangian systems. \\

The classical Helmholtz's result can be stated as follows : 

\begin{theorem}[Helmholtz's condition]
Let $O$ be a second order differential operator. The second order differential equation associated with $O$ can be written as a second order Euler-Lagrange equation if and only if all the Frechet derivatives of $O^{a,b}$ are self-adjoint for any real $a<b$.
\end{theorem}

We refer to \cite[Theorem 5.92, p.364-365]{olve} for a detailed proof. As remarked by J-P. Olver \cite[p.365]{olve}, the condition of self-adjointness of all the Frechet derivatives of $O^{a,b}$ for any real $a<b$ are often referred as the {\it Helmholtz's condition}. \\

Nevertheless, the Helmholtz's condition can be more explicitly formulated : 

\begin{lemma}
\label{helmholtz-continuous}
Let $O$ be a second order differential operator. The operator $O$ satisfies the Helmholtz's condition if and only if
\begin{equation}\label{helmcondcont}\tag{H$_{\text{cont}}$}
\forall a<b, \; \forall q \in \CC^2 ([a,b],\R), \; \dfrac{d}{dt} \left( \dfrac{\partial \OO}{\partial w} (q,\dot{q},\ddot{q},t) \right) = \dfrac{\partial \OO}{\partial v} (q,\dot{q},\ddot{q},t).
\end{equation}
\end{lemma}

\begin{proof}
Let $a<b$ and $q \in \CC^2 ([a,b],\R)$. One can easily obtain that:
\begin{equation}
\forall u \in \CC^2 ([a,b],\R), \; DO^{a,b}(q)(u) = \dfrac{\partial \OO}{\partial x} (q,\dot{q},\ddot{q},t) u + \dfrac{\partial \OO}{\partial v} (q,\dot{q},\ddot{q},t) \dot{u} + \dfrac{\partial \OO}{\partial w} (q,\dot{q},\ddot{q},t) \ddot{u}.
\end{equation}
Using integrations by part, it holds for any $u \in \CC^2 ([a,b],\R)$:
\begin{equation}
DO^{a,b}(q)^*(u) = \dfrac{\partial \OO}{\partial x} (q,\dot{q},\ddot{q},t) u - \dfrac{d}{dt} \left[ \dfrac{\partial \OO}{\partial v} (q,\dot{q},\ddot{q},t) u \right] + \dfrac{d^2}{dt^2} \left[ \dfrac{\partial \OO}{\partial w} (q,\dot{q},\ddot{q},t) u \right].
\end{equation}
Finally, we obtain that $DO^{a,b}(q) = DO^{a,b}(q)^*$ if and only if:
\begin{equation}
\dfrac{d}{dt} \left( \dfrac{\partial \OO}{\partial w} (q,\dot{q},\ddot{q},t) \right) = \dfrac{\partial \OO}{\partial v} (q,\dot{q},\ddot{q},t).
\end{equation}
This concludes the proof.
\end{proof}

\subsection{Second order discrete Euler-Lagrange equations} 
We give the discrete analogous definitions and results of the previous Section :

\begin{definition}\label{defdisclag}
A discrete Lagrangian functional, associated to a partition $\T \in \TT$, is defined by:
\begin{equation}
\fonction{\LL^{\T}}{\RN}{\R}{\Q}{ h \di \sum_{p=1}^N L_- \big( Q_p, (\Delta_- \Q)_p,t_p , h \big) + h \di \sum_{p=0}^{N-1} L_+ \big( Q_p, (-\Delta_+ \Q)_p,t_p , h \big),} 
\end{equation}
where $(L_-,L_+)$ is a couple of Lagrangians, \textit{i.e.} $ L_{\pm}$ are applications of the type:
\begin{equation}
\fonction{L_{\pm}}{\R^3 \times \R^+_*}{\R}{(x,v,t,\xi)}{L_{\pm} (x,v,t,\xi).}
\end{equation}
Let $\RN_0 := \{ \W \in \RN, \; W_0 = W_N = 0\}$ denote the set of \textit{discrete variations}. Then, $\Q \in \RN$ is said to be a \textit{discrete critical point} of $\LL^{\T}$ if for any discrete variation $\W$, $D\LL^{\T} (\Q)(\W)=0$.
\end{definition}

The discrete critical points of $\LL^{\T}$ are then characterized by :

\begin{theorem}\label{thmdisclag}
Let $(L_-,L_+)$ be a couple of Lagrangians and $\T \in \TT$. Let $\LL^{\T}$ be the associated discrete Lagrangian functional. Then, $\Q \in \RN$ is a discrete critical point of $\LL^{\T}$ if and only if $\Q$ is a solution of the second order discrete Euler-Lagrange equation given by :
\begin{multline}\label{elh}\tag{EL${}^{\T}$}
\di \frac{\partial L_-}{\partial x} ( \Q,\Delta_- \Q,\T,h )+\di \frac{\partial L_+}{\partial x} ( \Q,-\Delta_+ \Q,\T,h ) \\ + \Delta_+ \left( \frac{\partial L_-}{\partial v} ( \Q,\Delta_- \Q,\T,h) \right) - \Delta_- \left( \frac{\partial L_+}{\partial v} ( \Q,-\Delta_+ \Q,\T,h) \right)  = 0. 
\end{multline}
A discrete dynamical system governed equation (\ref{elh}) is called a \textit{second order discrete Lagrangian system}.
\end{theorem}

\begin{proof}
Let $\Q \in \RN$ and $\W \in \RN_0$. We have:
\begin{multline}\label{eq887}
D\LL^{\T} (\Q)(\W) = h \di \sum_{p=1}^N  \left[ \frac{\partial L_-}{\partial x} (\ast_p) W_p + \frac{\partial L_-}{\partial v} (\ast_p) (\Delta_- \W)_p \right] \\ + h \di \sum_{p=0}^{N-1} \left[  \frac{\partial L_+}{\partial x} (\ast \ast_p) W_p + \frac{\partial L_+}{\partial v} (\ast \ast_p) (-\Delta_+ \W)_p\right]  
\end{multline}
where $ \ast := (\Q,\Delta_- \Q,\T,h) $ and $\ast \ast := (\Q,-\Delta_+ \Q,\T,h) $. Let us remind the following \textit{discrete integrations by part}. For any $(\boldsymbol{F}, \boldsymbol{G}) \in \RN \times \RN_0 $, it holds:
\begin{equation}\label{eqdibp1}
\di \sum_{p=1}^{N} F_p (\Delta_- \boldsymbol{G})_p = \di \sum_{p=1}^{N-1} (\Delta_+ \boldsymbol{F})_p G_p \quad \text{and} \quad \di \sum_{p=0}^{N-1} F_p (\Delta_+ \boldsymbol{G})_p = \di \sum_{p=1}^{N-1} (\Delta_- \boldsymbol{F})_p G_p.
\end{equation}
Finally, combining \eqref{eq887} and \eqref{eqdibp1}, we obtain:
\begin{equation}
D\LL^{\T} (\Q)(\W) = h \di \sum_{p=1}^{N-1} \left[ \frac{\partial L_-}{\partial x} (\ast_p) + \frac{\partial L_+}{\partial x} (\ast \ast_p) + \Delta_+ \left( \frac{\partial L_-}{\partial v} (\ast) \right)_p - \Delta_- \left( \frac{\partial L_+}{\partial v} (\ast \ast) \right)_p \right] W_p,
\end{equation}
which concludes the proof.
\end{proof}

Let us consider an Euler-Lagrange equation $\eqref{el}$ defined on a real interval $[a,b]$ and $L$ the associated Lagrangian. Let us take for example $L_- (x,v,t, \xi) = L(x,v,t)$ and $L_+ = 0$. Considering a partition $\T \in \TT_{a,b}$, we obtain that $\LL^{\T}$ is a discrete version of $\LL^{a,b}$ and \eqref{elh} is obtained by a discrete variational principle on $\LL^{\T}$. The numerical scheme \eqref{elh} is then a so-called \textit{variational integrator} : it is a numerical scheme for \eqref{el} having the particularity of preserving its intrinsic Lagrangian structure at the discrete level. We refer to \cite{lubi,mars} for more details concerning variational integrators. We note that one can also use a centered version by taking $L_- (x,v,t, \xi) = L_+ (x,v,t, \xi) = L(x,v,t)/2$. 

\begin{example}\label{exnewtonsansfriction}
Let us consider the Newton's equation without friction $q+\ddot{q} = 0$ defined on a real interval $[a,b]$. It is a second order differential equation associated to $\OO (x,v,w,t) = x+w$ satisfying the continuous Helmholtz's condition \eqref{helmcondcont}. It corresponds to \eqref{el} with the quadratic Lagrangian $ L(x,v,t) = (x^2 - v^2)/2$. Considering a partition $\T \in \TT_{a,b}$, taking $L_- (x,v,t, \xi) = L(x,v,t)$ and $L_+ = 0$, we obtain the following discrete Euler-Lagrange equation:
\begin{equation}\label{eqnewtonsansfriction}
\forall p=1,...,N-1, \;   Q_p + \dfrac{Q_{p+1}- 2 Q_p + Q_{p-1}}{h^2} = 0.
\end{equation}
Let us note that \eqref{eqnewtonsansfriction} coincides with a direct discretization. Nevertheless, a direct discretization of an Euler-Lagrange equation do not lead necessary to a discrete Euler-Lagrange equation, see Example \ref{exmarchepas}. In this case, we say that the Lagrangian structure is not preserved.
\end{example}

\subsection{Formulation of the discrete Helmholtz's problem for second order finite differences equations}

We first remark that a second order discrete Euler-Lagrange equation is a second order finite differences equation (in the sense of Definition \ref{defSOFDE}):

\begin{proposition}\label{propelhSOFDE}
Let $(L_-,L_+)$ be a couple of Lagrangians. Then, the discrete Euler-Lagrange equation associated to $(L_-,L_+)$ is a second order finite differences equation defined by :
\begin{equation}\label{eqpropelhSOFDE1}
\fonction{\PP}{\R^5 \times \R^+_*}{\R}{(x,v_-,v_+,w,t,\xi)}{\di \sum_{i+j+k \geq 1} \Big[ A_{i,j,k} (x,v_-,t,\xi) v^i_+ + B_{i,j,k} (x,v_+,t,\xi) v^i_- \Big]w^j}
\end{equation}
where
\begin{equation}\label{eqpropelhSOFDE2}
A_{i,j,k}(x,v,t,\xi)=\delta_{(i,j,k)=(0,0,1)} \dfrac{\partial L_-}{\partial x}(x,v,t,\xi)- \dfrac{\xi^{i+j+k-1}}{i!j!k!}\dfrac{\partial^{i+j+k+1} L_-}{\partial x^i \partial v^{j+1} \partial t^k}(x,v,t,\xi) ,
\end{equation}
\begin{equation}\label{eqpropelhSOFDE3}
B_{i,j,k}(x,v,t,\xi)=\delta_{(i,j,k)=(0,0,1)} \dfrac{\partial L_+}{\partial x}(x,v,t,\xi)+ \dfrac{(-\xi)^{i+j+k-1}}{i!j!k!}\dfrac{\partial^{i+j+k+1} L_+}{\partial x^i \partial v^{j+1} \partial t^k}(x,v,t,\xi),
\end{equation}
and $\delta$ is the Kronecker symbol.
\end{proposition}

\begin{proof}
We have just to take a partition $\T \in \TT$ and to develop $\partial L_- / \partial v $ and $\partial L_+ / \partial v $ in power series.
\end{proof}

We finally formulate the following discrete version of the Helmholtz's problem : \\

\textit{\textbf{Discrete Helmholtz's problem for second order finite differences equations:} find a necessary and sufficient condition under which a second order finite differences equation can be written as a second order discrete Euler-Lagrange equation. Precisely, let $P$ be a second order finite differences operator. Our aim is to find a necessary and sufficient condition on $P$ under which there exists a couple of Lagrangian $(L_-,L_+)$ such that for any $\T \in \TT$ and any $\Q \in \RN$, we have:}
\begin{multline}
P^{\T}(\Q) = \di \frac{\partial L_-}{\partial x} ( \Q,\Delta_- \Q,\T,h )+\di \frac{\partial L_+}{\partial x} ( \Q,-\Delta_+ \Q,\T,h ) \\ + \Delta_+ \left( \frac{\partial L_-}{\partial v} ( \Q,\Delta_- \Q,\T,h) \right) - \Delta_- \left( \frac{\partial L_+}{\partial v} ( \Q,-\Delta_+ \Q,\T,h) \right).
\end{multline}

\section{Solution of the discrete Helmholtz's problem for second order finite differences equations}\label{section4}

\begin{theorem}\label{thmhelmdisc}
Let $P$ be a second order finite differences operator. Its associated second order finite differences equation can be written as a second order discrete Euler-Lagrange equation if and only if $P$ satisfies the following discrete Helmholtz's condition:
\begin{equation}\label{helmconddisc}\tag{H$_{\text{disc}}$}
\forall \T \in \TT, \; \forall \Q \in \RN, \; \forall p=2,...,N-1, \; \Delta_- \left( \dfrac{\partial \PP}{\partial w}(\star) \right)_p = \dfrac{\partial \PP}{\partial v_-}(\star_p) + \dfrac{\partial \PP}{\partial v_+}(\star_{p-1}),
\end{equation}
where $\star := \big( \Q, \Delta_- \Q, (-\Delta_+ \Q), (-\Delta_- \circ \Delta_- \Q),\T,h \big) $.
\end{theorem}

\begin{proof} See Sections \ref{section5} and \ref{section6}. \end{proof}

Let us note the similarity between the classical Helmholtz's condition \eqref{helmcondcont} and its discrete version \eqref{helmconddisc}. It is also important to note that, similarly to the continuous case, the discrete Helmholtz's condition \eqref{helmconddisc} corresponds to the self-adjointness of all the Frechet derivatives of $P^{\T}$ for any $ \T \in \TT $, see Section \ref{section5}.

\begin{example}
Let us take the second order finite differences equation \eqref{eqnewtonsansfriction} obtained in Example \ref{exnewtonsansfriction}. It is associated to the application $\PP (x,v_-,v_+,w,t,\xi) = x+w $ which satisfies the discrete Helmholtz's condition \eqref{helmconddisc}. It is expected because \eqref{eqnewtonsansfriction} is a discrete Euler-Lagrange equation by construction.
\end{example}

\begin{example}\label{exmarchepas}
Let us consider the differential equation $q + \sin (\dot{q}) \ddot{q} =0$ defined on a real interval $[a,b]$. It is a second order differential equation associated to the application $ \OO (x,v,w,t) = x + \sin (v) w$ satisfying \eqref{helmcondcont}. It corresponds to \eqref{el} associated to the Lagrangian $L(x,v,t) = (x^2/2) + \cos (v) $. Let us consider $\T \in \TT_{a,b}$ and define $\PP (x,v_-,v_+,w,t, \xi) = \OO \big( x,(1-\lambda)v_- + \lambda v_+ ,w,t \big) $ with $\lambda \in [0,1]$. Then, we obtain by direct discretization the following second order finite differences equation:
\begin{equation}\label{eqexmarchepas}
\forall p=1,...,N-1, \; Q_p + \sin \left( \dfrac{\lambda Q_{p+1}+(1-2\lambda)Q_p + (\lambda-1) Q_{p-1}}{h}  \right) \dfrac{Q_{p+1}-2Q_{p}+Q_{p-1}}{h^2} = 0.
\end{equation}
Let $P$ be the second order finite differences operator associated. Then, $P$ does not satisfy \eqref{helmconddisc} and consequently \eqref{eqexmarchepas} can not be written as a second order discrete Euler-Lagrange equation. This is an example of direct discretization of an Euler-Lagrange equation not leading to a discrete Euler-Lagrange equation. The numerical scheme \eqref{eqexmarchepas} does not preserve the Lagrangian structure of the differential equation at the discrete level.
\end{example}

\section{Discrete Helmholtz's condition and self-adjointness of Frechet derivatives of second order finite differences operators}
\label{section5}

In this section, we prove that a second order differences operator $P$ satisfies the discrete Helmholtz's condition \eqref{helmconddisc} if and only if all the Frechet derivatives of $ P^{\T}$ are self-adjoint for any $\T \in \TT$.

\subsection{Interpretation of the discrete Helmholtz's condition as self-adjointness of the Frechet derivative of second order finite differences operators} 

Let us define the following discrete version of the self-adjointness of a Frechet derivative for a differential operator :

\begin{definition}\label{defselfadjoint}
Let $P$ be a second order finite differences operator, $ \T \in \TT$ and $ \Q \in \RN$. We denote by $DP^{\T}(\Q)$ the Frechet derivative of $P^{\T}$ at the point $\Q$ and $DP^{\T}(\Q)^*$ the adjoint of $DP^{\T}(\Q)$ defined by :
\begin{equation}
\fonction{DP^{\T}(\Q)^*}{\RN}{\R^{N-3}}{\ZZ}{DP^{\T}(\Q)^*(\ZZ) = \big( DP^{\T}_p(\Q)^*(\ZZ) \big)_{p=2,...,N-2}}
\end{equation}
satisfying :
\begin{equation}
\forall (\W,\boldsymbol{Z}) \in \RN_{0,0} \times \RN, \; h \di \sum_{p=1}^{N-1} DP^{\T}_p(\Q)(\W) Z_p = h \di \sum_{p=2}^{N-2} DP^{\T}_p(\Q)^*(\boldsymbol{Z}) W_p .
\end{equation}
The Frechet derivative $DP^{\T} (\Q) $ is said to be \textit{self-adjoint} if for any $p=2,...,N-2$, $DP^{\T}_p (\Q) = DP^{\T}_p (\Q)^*$.
\end{definition}

A simple calculation leads to the following result :
 
\begin{proposition}
\label{propfrechderiv}
Let $P$ be a second order finite differences operator, $ \T \in \TT$ and $ \Q \in \RN$. Then, for any $p=1,...,N-1$ and any $\W \in \RN$, the Frechet derivative of $P$ is given by
\begin{equation}\label{eqpropfrechderiv}
DP^{\T}_p(\Q)(\W) = \dfrac{\partial \PP}{\partial x} (\star_p) W_p + \dfrac{\partial \PP}{\partial v_-} (\star_p) (\Delta_- \W)_p + \dfrac{\partial \PP}{\partial v_+} (\star_p) (-\Delta_+ \W)_p + \dfrac{\partial \PP}{\partial w} (\star_p) (-\Delta_+ \circ \Delta_- \W)_p,
\end{equation}
where $\star := \big( \Q, \Delta_- \Q, (-\Delta_+ \Q), (-\Delta_- \circ \Delta_- \Q),\T,h \big) $.
\end{proposition}

Applying Lemmas \ref{dlf} and \ref{dibp} in Proposition \ref{propfrechderiv}, we then deduce the explicit form of  the adjoint of the Frechet derivative of a second order difference operator :

\begin{proposition}\label{propadjfrechderiv}
Let $P$ be a second order finite differences operator, $ \T \in \TT$ and $ \Q \in \RN$. Then, for any $p=2,...,N-2$ and any $\W \in \RN$, $ DP^{\T}_p(\Q)^*(\W) $ is equal to:
\begin{equation}\label{eqpropadjfrechderiv}
\begin{array}{c}
\left[ \dfrac{\partial \PP}{\partial x} (\star_p) - \left( -\Delta_+ \dfrac{\partial \PP}{\partial v_-} (\star) \right)_p - \left( \Delta_- \dfrac{\partial \PP}{\partial v_+} (\star) \right)_p + \left( -\Delta_+ \circ \Delta_- \dfrac{\partial \PP}{\partial w} (\star) \right)_p \right] W_p \\
+ \left[ \left(\Delta_- \dfrac{\partial \PP}{\partial w} (\star) \right)_p - \dfrac{\partial \PP}{\partial v_+} (\star_{p-1}) \right] (\Delta_- \W)_p \\
+ \left[ \left( - \Delta_+ \dfrac{\partial \PP}{\partial w} (\star) \right)_p - \dfrac{\partial \PP}{\partial v_-} (\star_{p+1}) \right] (-\Delta_+ \W)_p +\dfrac{\partial \PP}{\partial w} (\star_p) (-\Delta_+ \circ \Delta_- \W)_p,
\end{array}
\end{equation}
where $\star := \big( \Q, \Delta_- \Q, (-\Delta_+ \Q), (-\Delta_- \circ \Delta_- \Q),\T , h \big) $.
\end{proposition}

The main result of this Section is the following explicit characterization of second order finite differences operators $P$ whose all Frechet derivatives are self-adjoint for any $\T \in \TT$ :

\begin{theorem}\label{thmcaractadj}
Let $P$ be a second order finite differences operator. Then, $DP^{\T}(\Q)$ is self adjoint for any $\T \in \TT$ and any $\Q \in \RN$ if and only if $P$ satisfies the discrete Helmholtz's condition \eqref{helmconddisc}.
\end{theorem}

\begin{proof}
Let $\T \in \TT$ and $\Q \in \RN$. According to Proposition \ref{propadjfrechderiv}, we have that $DP^{\T}(\Q)$ is self adjoint if and only if the right term of \eqref{eqpropfrechderiv} is equal to \eqref{eqpropadjfrechderiv} for any $\W \in \R^{N+1}$ and any $p=2,...,N-2$. As a consequence, $DP^{\T}(\Q)$ is self adjoint if and only for any $p=2,...,N-2$, the following system is satisfied :
\begin{eqnarray}
& \left( -\Delta_+ \circ \Delta_- \dfrac{\partial \overline{P}}{\partial w} (\star) \right)_p - \left( -\Delta_+ \dfrac{\partial \overline{P}}{\partial v_-} (\star) \right)_p - \left( \Delta_- \dfrac{\partial \overline{P}}{\partial v_+} (\star) \right)_p  = 0, &\\
&\left( - \Delta_+ \dfrac{\partial \overline{P}}{\partial w} (\star) \right)_p - \dfrac{\partial \overline{P}}{\partial v_-} (\star_{p+1}) = \dfrac{\partial \overline{P}}{\partial v_+} (\star_p), &\\
&\left(\Delta_- \dfrac{\partial \overline{P}}{\partial w} (\star) \right)_p - \dfrac{\partial \overline{P}}{\partial v_+} (\star_{p-1}) = \dfrac{\partial \overline{P}}{\partial v_-} (\star_p) , &
\end{eqnarray}
It is easy to verify that this system is equivalent to :
\begin{equation}
\left(\Delta_- \dfrac{\partial \overline{P}}{\partial w} (\star) \right)_p = \dfrac{\partial \overline{P}}{\partial v_-} (\star_{p}) + \dfrac{\partial \overline{P}}{\partial v_+} (\star_{p-1}),
\end{equation}
for any $p=2,...,N-1$. This concludes the proof.
\end{proof}

\section{Proof of Theorem \ref{thmhelmdisc}}\label{section6}

\subsection{Sufficient condition}
Let $P$ be a second order finite differences operator associated to a second order discrete Euler-Lagrange equation. Let $(L_-,L_+) $ denotes the associated couple of Lagrangians. According to Proposition \ref{propelhSOFDE}, we have that $\PP$ satisfies \eqref{eqpropelhSOFDE1}. With a simple calculation, we can prove that for any $\T \in \TT$, any $\Q \in \RN$ and any $p=1,...,N-1$, the three following equalities hold:
\begin{equation}
\dfrac{\partial \PP}{\partial w} (\star_p) = - \dfrac{\partial^2 L_-}{\partial v^2}\big( Q_{p+1},(\Delta_- \Q)_{p+1},t_{p+1},h\big) - \dfrac{\partial^2 L_+}{\partial v^2}\big( Q_{p-1},(-\Delta_+ \Q)_{p-1},t_{p-1},h\big),
\end{equation}
\begin{multline}
\dfrac{\partial \PP}{\partial v_-} (\star_p) = \dfrac{\partial^2 L_-}{\partial x \partial v}\big( Q_{p},(\Delta_- \Q)_{p},t_{p},h\big) - \dfrac{\partial^2 L_+}{\partial x \partial v}\big( Q_{p-1},(-\Delta_+ \Q)_{p-1},t_{p-1},h\big)  \\  + \Delta_+ \left( \dfrac{\partial^2 L_-}{\partial v^2} ( \Q,\Delta_- \Q,\T,h ) \right)_p
\end{multline}
\begin{multline}
\dfrac{\partial \PP}{\partial v_+} (\star_p) = \dfrac{\partial^2 L_+}{\partial x \partial v}\big( Q_{p},(-\Delta_+ \Q)_{p},t_{p},h\big) - \dfrac{\partial^2 L_-}{\partial x \partial v}\big( Q_{p+1},(\Delta_- \Q)_{p+1},t_{p+1},h\big)  \\  - \Delta_- \left( \dfrac{\partial^2 L_+}{\partial v^2} ( \Q,-\Delta_+ \Q,\T,h ) \right)_p,
\end{multline}
where $\star := \big( \Q,\Delta_- \Q,(-\Delta_+ \Q),(-\Delta_+ \circ \Delta_- \Q),\T , h\big)$. Finally, from these three equalities, we prove that $P$ satisfies the discrete Helmholtz's condition \eqref{helmconddisc}.

\subsection{Necessary condition} 
Let $P$ be a second order finite differences operator satisfying the discrete Helmholtz's condition \eqref{helmconddisc}. The proof is based on the following proposition:

\begin{proposition}
\label{propnececond}
Let $P$ be a second order finite differences operator satisfying the discrete Helmholtz's condition \eqref{helmconddisc}. Let $L_1$ be the following augmented Lagrangian:
\begin{equation}
\fonction{L_1}{\R^5 \times \R^+_*}{\R}{(x,v_-,v_+,w,t,\xi)}{x \di \int_0^1 \PP (\lambda x,\lambda v_-,\lambda v_+,\lambda w,t,\xi) \; d\lambda}
\end{equation}
and, for any $\T \in \TT$, let $\LL^{\T}_1 $ denote the following augmented Lagrangian functional
\begin{equation}
\fonction{\LL^{\T}_1}{\RN}{\R}{\Q}{h \di \sum_{p=1}^{N-1} L_1 \big( Q_p , (\Delta_- \Q)_p , (-\Delta_+ \Q)_p, (-\Delta_+ \circ \Delta_- \Q)_p ,t_p ,h \big). }
\end{equation}
Then,
\begin{enumerate}
\item for any $ \T \in \TT$ and any $ (\Q,\W) \in \RN \times \RN_{0,0}$, it holds:
\begin{equation}
D\LL^{\T}_1 (\Q)(\W) = h \di \sum_{p=2}^{N-2} P^{\T}_p (\Q) W_p,
\end{equation}
\item there exists a couple of Lagrangians $(L_-,L_+) $ such that for any $\T \in \TT$, any $\Q \in \RN $, it holds:
\begin{multline}
L_1 \big( \Q, \Delta_- \Q, (-\Delta_+ \Q), (-\Delta_- \circ \Delta_- \Q),\T,h \big) \\ = L_- \big( \Q, \Delta_- \Q,\T,h \big) + L_+ \big( \Q , (-\Delta_+ \Q) , \T,h \big).
\end{multline}
\end{enumerate}
\end{proposition}

\begin{proof}
$1.$ Let $\T \in \TT$ and $ (\Q,\W) \in \RN \times \RN_{0,0}$. We have:
\begin{equation}
\LL^{\T}_1 (\Q) = h \di \sum_{p=1}^{N-1} Q_p \di \int_0^1 P^{\T}_p(\lambda \Q) \; d\lambda.
\end{equation}
Thus:
\begin{equation}\label{eq443}
D\LL^{\T}_1(\Q)(\W) = h \di \sum_{p=1}^{N-1} W_p \di \int_0^1 P^{\T}_p (\lambda \Q) \; d\lambda + h \di \sum_{p=1}^{N-1} Q_p \di \int_0^1 DP^{\T}_p (\lambda \Q)(\lambda \W) \; d\lambda.
\end{equation}
As $P$ satisfies the discrete Helmholtz's condition \eqref{helmconddisc} and according to Theorem \ref{thmcaractadj}, $DP^{\T}(\lambda \Q)$ is self-adjoint. Using Definition \ref{defselfadjoint}, the following equality holds:
\begin{equation}\label{eq444}
h \di \sum_{p=1}^{N-1} DP^{\T}_p (\lambda \Q)(\lambda \W) Q_p = h \di \sum_{p=2}^{N-2} \lambda DP^{\T}_p (\lambda \Q)(\Q) W_p .
\end{equation}
Then, from Equalities \eqref{eq443} and \eqref{eq444}, we have
\begin{equation}
D\LL^{\T}_1(\Q)(\W) = h \di \sum_{p=2}^{N-2} W_p \di \int_0^1 P^{\T}_p (\lambda \Q) \; d\lambda + h \di \sum_{p=2}^{N-2} W_p \di \int_0^1 \lambda  DP^{\T}_p (\lambda \Q)(\Q)\; d\lambda
\end{equation}
As a consequence, using the equality $\partial / \partial \lambda \big( P^{\T}_p(\lambda \Q) \big) = DP^{\T}_p (\lambda \Q)(\Q)$ and an integration by part with respect to $\lambda$ on the second integral, we obtain :
\begin{equation}
D\mathcal{L}^{\T}_1(\Q)(\W) = h \di \sum_{p=2}^{N-2} P^{\T}_p (\Q) W_p.
\end{equation}
$2.$ Since $P$ satisfies the discrete Helmholtz's condition \eqref{helmconddisc}, we have for any $\T \in \TT$ and any $\Q \in \RN$ :
\begin{equation}
\forall p=2,...,N-1, \; \dfrac{1}{h} \left( \dfrac{\partial \PP}{\partial w} (\star_p) - \dfrac{\partial \PP}{\partial w} (\star_{p-1}) \right) = \dfrac{\partial \PP}{\partial v_-} (\star_p) + \dfrac{\partial \PP}{\partial v_+} (\star_{p-1}),
\end{equation}
where $\star := \big( \Q, \Delta_- \Q, (-\Delta_+ \Q), (-\Delta_- \circ \Delta_- \Q),\T,h \big) $. As it is true for any $\Q \in \RN$, we can differentiate the previous equality with respect to $Q_{p-2}$ and $Q_{p+1}$. It leads to the two following equalities holding for any $\T \in \TT$, any $\Q \in \RN$ and any $p=2,...,N-1$ :
\begin{equation}
\dfrac{1}{h} \left( \dfrac{1}{h} \dfrac{\partial^2 \PP}{\partial v_- \partial w} (\star_{p-1}) - \dfrac{1}{h^2} \dfrac{\partial^2 \PP}{\partial w^2} (\star_{p-1}) \right) = -\dfrac{1}{h} \dfrac{\partial^2 \PP}{\partial v_- \partial v_+} (\star_{p-1}) + \dfrac{1}{h^2} \dfrac{\partial^2 \PP}{\partial w \partial v_+} (\star_{p-1})
\end{equation}
and
\begin{equation}
\dfrac{1}{h} \left( \dfrac{1}{h} \dfrac{\partial^2 \PP}{\partial v_+ \partial w} (\star_p) + \dfrac{1}{h^2} \dfrac{\partial^2 \PP}{\partial w^2} (\star_p) \right) = \dfrac{1}{h} \dfrac{\partial^2 \PP}{\partial v_+ \partial v_-} (\star_p) + \dfrac{1}{h^2} \dfrac{\partial^2 \PP}{\partial w \partial v_-} (\star_p).
\end{equation}
Finally, we have for any $\T \in \TT$, any $\Q \in \RN$ and any $p=1,...,N-1$ :
\begin{equation}\label{eq888}
\dfrac{\partial^2 \PP}{\partial v_+ \partial v_-} (\star_p) + \dfrac{1}{h} \left(\dfrac{\partial^2 \PP}{\partial w \partial v_-} (\star_p) - \dfrac{\partial^2 \PP}{\partial v_+ \partial w} (\star_p) \right) +\dfrac{1}{h^2} \dfrac{\partial^2 \PP}{\partial w^2} (\star_p) = 0.
\end{equation}
Since Equality \eqref{eq888} is true for any $\T \in \TT$, any $\Q \in \RN$ and any $p=1,...,N-1$, we have for any $(x,y,z,t,\xi) \in \R^4 \times \R^+_* $:
\begin{multline}\label{eqproof}
\dfrac{\partial^2 \PP}{\partial v_+ \partial v_-} \Big(x,y,z,\dfrac{z-y}{\xi},t,\xi\Big) + \dfrac{1}{h^2} \dfrac{\partial^2 \PP}{\partial w^2} \Big(x,y,z,\dfrac{z-y}{\xi},t,\xi \Big) \\ +\dfrac{1}{h} \left(\dfrac{\partial^2 \PP}{\partial w \partial v_-} \Big(x,y,z,\dfrac{z-y}{\xi},t,\xi \Big) - \dfrac{\partial^2 \PP}{\partial v_+ \partial w} \Big(x,y,z,\dfrac{z-y}{\xi},t,\xi \Big) \right)= 0.
\end{multline}
Let us define:
\begin{equation}
\fonction{\ell}{\R^4 \times \R^+_*}{\R}{(x,y,z,t,\xi)}{\PP \Big( x,y,z,\dfrac{z-y}{\xi},t,\xi \Big).}
\end{equation}
According to \eqref{eqproof}, it holds:
\begin{equation}
\forall (x,y,z,t,\xi) \in \R^4 \times \R^+_*, \;  \dfrac{\partial^2 \ell}{\partial z \partial y}(x,y,z,t,\xi) = 0.
\end{equation}
We deduce that the variables $y$ and $z$ are separable in $\ell$. Precisely, there exist two functions $\alpha, \; \beta : \R^3 \times \R^+_* \longrightarrow \R$ such that:
\begin{equation}
\forall (x,y,z,t,\xi) \in \R^4 \times \R^+_*, \; \ell (x,y,z,t,\xi) = \alpha (x,y,t,\xi) + \beta (x,z,t,\xi).
\end{equation}
Finally, we have for any $\T \in \TT$, any $\Q \in \RN$ and any $p=1,...,N-1$:
\begin{eqnarray}
L_1 (\star_p) & = & Q_p \left( \di \int_0^1 \PP \big( \lambda Q_p, \lambda (\Delta_- \Q)_p, \lambda (-\Delta_+ \Q)_p,\lambda (-\Delta_+ \circ \Delta_- \Q)_p,t_p,h \big) \; d\lambda \right) \\
& = & Q_p \left( \di \int_0^1 \ell \big( \lambda Q_p, \lambda (\Delta_- \Q)_p, \lambda (-\Delta_+ \Q)_p,t_p,h \big) \; d\lambda \right) \\
& = & Q_p \left( \di \int_0^1 \alpha \big( \lambda Q_p, \lambda (\Delta_- \Q)_p,t_p,h \big) + \beta \big( \lambda Q_p, \lambda (-\Delta_+ \Q)_p,t_p,h \big) \; d\lambda \right) \\
& = & L_- \big( Q_p , (\Delta_- \Q)_p , t_p,h \big) + L_+ \big( Q_p , (-\Delta_+ \Q)_p , t_p,h \big),
\end{eqnarray}
where
\begin{equation}
\fonction{L_-}{\R^3 \times \R^+_*}{\R}{(x,v,t,\xi)}{x \di \int_0^1 \alpha(\lambda x,\lambda v,t,\xi) \; d\lambda}
\end{equation}
and 
\begin{equation}
\fonction{L_+}{\R^3 \times \R^+_*}{\R}{(x,v,t,\xi)}{x \di \int_0^1 \beta(\lambda x,\lambda v,t,\xi) \; d\lambda.}
\end{equation}
This concludes the proof.
\end{proof}

The proof of Theorem \ref{thmhelmdisc} can be deduced from Proposition \ref{propnececond} as follows. For any $\T \in \TT$, we define the following discrete Lagrangian functional :
\begin{equation}
\label{eq555}
\fonction{\LL^{\T}}{\RN}{\R}{\Q}{h \di \sum_{p=1}^N L_- \big( Q_p, (\Delta_- \Q)_p,t_p ,h\big) + h \di \sum_{p=0}^{N-1} L_+ \big( Q_p, (-\Delta_+ \Q)_p,t_p,h \big)}
\end{equation}
where $(L_-,L_+)$ is the couple of Lagrangians given in the point $2$ of Proposition \ref{propnececond}. We easily verify that for any $\T \in \TT$ and any $\Q \in \RN$ we have 
\begin{equation}
\LL^{\T}(\Q) = \LL^{\T}_1 (\Q) + h L_- \left( Q_N, \dfrac{Q_N - Q_{N-1}}{h},t_N,h \right) + h L_+ \left( Q_0, \dfrac{Q_1 - Q_{0}}{h},t_0,h \right),
\end{equation}
where $\LL^{\T}_1$ is defined in Proposition \ref{propnececond}. As a consequence, for any $\T \in \TT$, any $\Q \in \RN$ and any $\W \in \RN_{0,0}$ we obtain that
\begin{equation}\label{eq111}
D\LL^{\T} (\Q)(\W) = D\LL^{\T}_1 (\Q)(\W) = h \di \sum_{p=2}^{N-2} P^{\T}_p(\Q) W_p .
\end{equation}
However, using the same method than in the proof of Theorem \ref{thmdisclag}, we deduce from Equation \eqref{eq555} that for any $\T \in \TT$, any $\Q \in \RN$ and any $\W \in \RN_{0,0}$:
\begin{equation}\label{eq112}
D\LL^{\T} (\Q)(\W) = h \di \sum_{p=2}^{N-2} \left[ \frac{\partial L_-}{\partial x} (\ast_p) + \frac{\partial L_+}{\partial x} (\ast \ast_p) + \Delta_+ \left( \frac{\partial L_-}{\partial v} (\ast) \right)_p - \Delta_- \left( \frac{\partial L_+}{\partial v} (\ast \ast) \right)_p \right] W_p,
\end{equation}
where $ \ast := (\Q,\Delta_- \Q,\T,h) $ and $\ast \ast := (\Q,-\Delta_+ \Q,\T,h) $. Combining Equalities \eqref{eq111} and \eqref{eq112}, we conclude that for any $\T \in \TT$, any $\Q \in \RN$ and any $p=2,...,N-2$:
\begin{equation}\label{eq886}
P^{\T}_p (\Q) = \frac{\partial L_-}{\partial x} (\ast_p) + \frac{\partial L_+}{\partial x} (\ast \ast_p) + \Delta_+ \left( \frac{\partial L_-}{\partial v} (\ast) \right)_p - \Delta_- \left( \frac{\partial L_+}{\partial v} (\ast \ast) \right)_p.
\end{equation}
In order to finish the proof of Theorem \ref{thmhelmdisc}, we have just to prove that Equality \eqref{eq886} is still true for $p=1$ and $p=N-1$. We only prove it for $p=N-1$. The case $p=1$ can be proved in a similar way. \\

Let $\T \in \TT$ and $\Q \in \RN$. We denote by $\sigma (\T) = (t_{p+1})_{p=0,...,N} \in \TT$ and $\sigma (\Q) = (Q_{p+1})_{p=0,...,N} \in \RN$ where $t_{N+1} := t_N + h$ and $Q_{N+1} := 0$. From Equality \eqref{eq886}, we obtain
\begin{multline}
P^{\T}_{N-1} \big( \Q \big) = P^{\sigma(\T)}_{N-2} \big( \sigma (\Q) \big) = \di \frac{\partial L_-}{\partial x} \big(\sigma (\ast)_{N-2} \big) + \frac{\partial L_+}{\partial x} \big( \sigma (\ast \ast)_{N-2} \big) \\ + \Delta_+ \left( \frac{\partial L_-}{\partial v} \big( \sigma (\ast) \big) \right)_{N-2} - \Delta_- \left( \frac{\partial L_+}{\partial v} \big( \sigma (\ast \ast) \big) \right)_{N-2} ,
\end{multline}
where $ \sigma (\ast) := \big(\sigma(\Q),\Delta_- \sigma(\Q),\sigma(\T),h \big)$ and $ \sigma (\ast \ast) := \big(\sigma(\Q),-\Delta_+ \sigma(\Q),\sigma(\T),h \big)$. Consequently, we have:
\begin{multline}
P^{\T}_{N-1} \big( \Q \big) = \frac{\partial L_-}{\partial x} (\ast_{N-1}) + \frac{\partial L_+}{\partial x} (\ast \ast_{N-1}) + \Delta_+ \left( \frac{\partial L_-}{\partial v} (\ast) \right)_{N-1} - \Delta_- \left( \frac{\partial L_+}{\partial v} (\ast \ast) \right)_{N-1}.
\end{multline}
This concludes the proof of Theorem \ref{thmhelmdisc}.

\section{Characterization of null (couples of) Lagrangians}
In this section, we are interested in the second part of the Helmholtz's problem both in the continuous and discrete cases. Precisely, once the Helmholtz's condition is satisfied, can we characterize all the possible (couples of) Lagrangians leading to the same second order (discrete) Euler-Lagrange equation ? %

\subsection{Reminder of the continuous case} Let $L^1$, $L^2$ be two Lagrangians. They are said to be \textit{equivalent} if they lead to the same second order Euler-Lagrange equation. In this case, we denote by $L^1 \sim L^2$. The linearity of the Euler-Lagrange equation with respect to its associated Lagrangian implies that $\sim$ defines an equivalence relation on the set of Lagrangians (see also \cite[p.356]{olve}). \\

Hence, the aim is to characterize the equivalence class of $0$. If a Lagrangian $L$ belongs to the equivalence class of $0$, then it leads to a null second order Euler-Lagrange equation in the sense that every curves $q$ are solutions. In this case, $L$ is said to be a \textit{null Lagrangian} (see \cite[p.247-249]{olve}). We refer to \cite[p.248,Theorem 4.7]{olve} for a detailed proof of the following result :

\begin{theorem}\label{thmnullcont}
Let $L$ be a Lagrangian. $L$ is a null Lagrangian if and only if there exist two functions $f : \R^2 \longrightarrow \R$ and $g : \R \longrightarrow \R$ such that :
\begin{equation}
\forall a <b, \; \forall q \in \CC^2([a,b],\R), \; L(q,\dot{q},t) = \dfrac{d}{dt} \big( f (q,t) \big) + g(t).
\end{equation}
\end{theorem}

Let us note that the previous theorem is usually stated with $g=0$. Indeed, we have just to add an anti-derivative of $g$ to $f$. However, in the next section, we can prove the discrete version of Theorem \ref{thmnullcont} only with this presentation.

\subsection{The discrete case} We give the following discrete versions of the definitions and results of the previous section. 

\begin{definition}
Let $(L^1_-,L^1_+) $ and $(L^2_-,L^2_+) $ be two couples of Lagrangians. We say that they are \textit{equivalent} if they lead to the same discrete second order Euler-Lagrange equation \eqref{elh}. In this case, we denote by $(L^1_-,L^1_+) \sim (L^2_-,L^2_+)$. The linearity of the discrete Euler-Lagrange equation with respect to its associated couple of Lagrangians implies that $\sim$ defines an equivalence relation on the set of couple of Lagrangians. If a couple of Lagrangians $(L_-,L_+)$ belongs to the equivalence class of $0$, then it leads to a null second order discrete Euler-Lagrange equation in the sense that every discrete curves $\Q$ are solutions. In this case, $(L_-,L_+)$ is said to be a \textit{null couple of Lagrangians}.
\end{definition}

The next result characterize the set of null couple of Lagrangians. 

\begin{theorem}\label{thmnulldisc}
Let $(L_-,L_+)$ be a couple of Lagrangians. $(L_-,L_+)$ is a null couple of Lagrangians if and only if there exist two functions $f : \R^2 \times \R^+_* \longrightarrow \R$ and $g : \R \times \R^+_* \longrightarrow \R$ such that for any $\T \in \TT$, any $\Q \in \RN$ and any $p=1,...,N$:
\begin{equation}\label{eq999}
L_- \big(Q_{p},(\Delta_- Q)_p,t_p,h \big) + L_+ \big(Q_{p-1},(-\Delta_+ Q)_{p-1},t_{p-1},h \big) = \Delta_- \big( f (\Q,\T,h) \big)_p +g(t_p,h).
\end{equation}
\end{theorem}

\begin{proof}
\textbf{\textit{Sufficiency}}. Let us assume that Equation \eqref{eq999} is true and let $\LL^{\T}$ denote the discrete Lagrangian functional associated to $(L_-,L_+)$ and to a partition $\T \in \TT$. Then, we have for any $\T \in \TT$ and any $\Q \in \RN$:
\begin{eqnarray*}
\LL^{\T} (\Q) & = & h \di \sum_{p=1}^N L_- \big( Q_p, (\Delta_- \Q)_p,t_p , h \big) + h \di \sum_{p=0}^{N-1} L_+ \big( Q_p, (-\Delta_+ \Q)_p,t_p , h \big) \\
& = & h \di \sum_{p=1}^N \Big[ L_- \big( Q_p, (\Delta_- \Q)_p,t_p , h \big) + L_+ \big(Q_{p-1},(-\Delta_+ Q)_{p-1},t_{p-1},h \big) \Big] \\
& = & h \di \sum_{p=1}^N \Big[ \Delta_- \big( f (\Q,\T,h) \big)_p + g(t_p,h) \Big]\\
& = & f (Q_N,t_N,h) - f (Q_0,t_0,h) + h \di \sum _{p=1}^N g(t_p,h).
\end{eqnarray*}
Consequently, since the set of discrete variations is $\RN_0$, every discrete curves $\Q \in \RN$ are discrete critical points of $\LL^{\T}$ and then the discrete Euler-Lagrange equation associated is null. Then, $(L_-,L_+)$ is a null couple of Lagrangian. \\

\textbf{\textit{Necessity}}. We assume that $(L_-,L_+)$ is a null couple of Lagrangian. Then, for any $\T \in \TT$ and any $\Q \in \RN$, it holds:
\begin{multline}\label{eqprooffin4}
\di \frac{\partial L_-}{\partial x} ( \Q,\Delta_- \Q,\T,h )+\di \frac{\partial L_+}{\partial x} ( \Q,-\Delta_+ \Q,\T,h ) \\ + \Delta_+ \left( \frac{\partial L_-}{\partial v} ( \Q,\Delta_- \Q,\T,h) \right) - \Delta_- \left( \frac{\partial L_+}{\partial v} ( \Q,-\Delta_+ \Q,\T,h) \right)  = 0. 
\end{multline}
Then, let us define:
\begin{equation}
\fonction{\ell_-}{\R^3 \times \R^+_*}{\R}{(x_1,x_2,t,\xi)}{L_- \Big( x_1,\dfrac{x_1-x_2}{\xi},t, \xi \Big)}
\end{equation}
and
\begin{equation}
\fonction{\ell_+}{\R^3 \times \R^+_*}{\R}{(x_1,x_2,t,\xi)}{L_+ \Big( x_1,\dfrac{x_2-x_1}{\xi},t, \xi \Big).}
\end{equation}
Since Equality \eqref{eqprooffin4} is true for any $\T \in \TT$ and any $\Q \in \RN$, we have for any $(x,y,z,t,\xi) \in \R^4 \times \R^+_*$:
\begin{equation}
\dfrac{\partial \ell_-}{\partial x_1} (x,y,t,\xi)+\dfrac{\partial \ell_-}{\partial x_2} (z,x,t+\xi,\xi) +\dfrac{\partial \ell_+}{\partial x_1} (x,z,t,\xi)+\dfrac{\partial \ell_+}{\partial x_2}(y,x,t-\xi,\xi)=0.
\end{equation}
Then, by differentiating the previous equality with respect to $y$ or to $z$, we obtain for any $(x,y,t,\xi) \in \R^3 \times \R^+_*$:
\begin{equation}
\dfrac{\partial^2 \ell_-}{\partial x_1 \partial x_2} (x,y,t,\xi) + \dfrac{\partial^2 \ell_+}{\partial x_1 \partial x_2} (y,x,t-\xi,\xi) = 0.
\end{equation}
Consequently, there exist two functions $\alpha, \beta : \R^2 \times \R^+_* \longrightarrow \R$ such that for any $(x,y,t,\xi) \in \R^3 \times \R^+_*$:
\begin{equation}\label{prooffin2}
\ell_- (x,y,t,\xi) + \ell_+ (y,x,t-\xi,\xi) = \alpha (x,t,\xi) + \beta (y,t,\xi).
\end{equation}
Moreover, let us denote by $\LL^{\T}$ the discrete Lagrangian functional associated to $(L_-,L_+)$ and to a partition $\T \in \TT$. Then, we have for any $\T \in \TT$ and any $\Q \in \RN$:
\begin{eqnarray*}
\LL^{\T} (\Q) & = & h \di \sum_{p=1}^N L_- \big( Q_p, (\Delta_- \Q)_p,t_p , h \big) + h \di \sum_{p=0}^{N-1} L_+ \big( Q_p, (-\Delta_+ \Q)_p,t_p , h \big) \\
& = & h \di \sum_{p=1}^N \Big[ L_- \big( Q_p, (\Delta_- \Q)_p,t_p , h \big) + L_+ \big(Q_{p-1},(-\Delta_+ Q)_{p-1},t_{p-1},h \big) \Big] \\
& = & h \di \sum_{p=1}^N \Big[ \ell_- (Q_p,Q_{p-1},t_p,h) + \ell_+ (Q_{p-1},Q_p,t_{p-1},h) \Big] \\
& = & h \di \sum_{p=1}^N \Big[ \alpha (Q_p,t_p,h) + \beta (Q_{p-1},t_{p},h) \Big].
\end{eqnarray*}
With a discrete calculus of variations, we obtain for any $\T \in \TT$, any $\Q \in \RN$ and any $\W \in \RN_0$:
\begin{equation}
D\LL^{\T}(\Q)(\W) = \di \sum_{p=1}^{N-1} \Big[ \dfrac{\partial \alpha}{\partial x_1} (Q_p,t_p,h) + \dfrac{\partial \beta}{\partial x_1} (Q_p,t_{p+1},h) \Big] W_p.
\end{equation}
Since $(L_-,L_+)$ is a null couple of Lagrangian, $D\LL^{\T}(\Q)(\W) = 0$ for any $\T \in \TT$, any $\Q \in \RN$ and any $\W \in \RN_0$. Consequently, for any $\T \in \TT$ and any $\Q \in \RN$, it holds:
\begin{equation}
\dfrac{\partial \alpha}{\partial x_1} (Q_p,t_p,h) + \dfrac{\partial \beta}{\partial x_1} (Q_p,t_{p+1},h) = 0.
\end{equation}
Consequently, we have for any $(x,t,\xi) \in \R^2 \times \R^+_*$:
\begin{equation}
\dfrac{\partial \alpha}{\partial x_1} (x,t,\xi) + \dfrac{\partial \beta}{\partial x_1} (x,t+\xi,\xi) = 0.
\end{equation}
Hence, there exists a function $\gamma : \R \times \R^+_* \longrightarrow \R$ such that:
\begin{equation}
\forall (x,t,\xi) \in \R^2 \times \R^+_*, \; \alpha (x,t,\xi) + \beta (x,t+\xi,\xi) = \gamma (t,\xi).
\end{equation}
Then, according to Equality \eqref{prooffin2}, we have for any $(x,y,t,\xi) \in \R^3 \times \R^+_*$:
\begin{equation}
\ell_- (x,y,t,\xi) + \ell_+ (y,x,t-\xi,\xi) = \alpha (x,t,\xi) - \alpha(y,t-\xi,\xi) + \gamma (t-\xi,\xi).
\end{equation}
Thus, we have for any $\T \in \TT$, any $\Q \in \RN$ and any $p=1,...,N$:
\begin{multline}
L_- \big( Q_p, (\Delta_- \Q)_p,t_p , h \big) + L_+ \big(Q_{p-1},(-\Delta_+ Q)_{p-1},t_{p-1},h \big) \\
\begin{array}{ll}
= & \ell_- (Q_p,Q_{p-1},t_p,h) + \ell_+ (Q_{p-1},Q_p,t_{p-1},h) \\ 
= & \alpha (Q_p,t_p,h) - \alpha(Q_{p-1},t_{p-1},h) + \gamma (t_{p-1},h)
\end{array}
\end{multline}
Hence, let us define:
\begin{equation}
\fonction{f}{\R^2 \times \R^+_*}{\R}{(x,t,\xi)}{\xi \alpha (x,t,\xi)}
\end{equation}
and
\begin{equation}
\fonction{g}{\R \times \R^+_*}{\R}{(t,\xi)}{\gamma(t-\xi,\xi).}
\end{equation}
Then, we have for any $\T \in \TT$, any $\Q \in \RN$ and any $p=1,...,N$:
\begin{equation}
L_- \big( Q_p, (\Delta_- \Q)_p,t_p , h \big) + L_+ \big(Q_{p-1},(-\Delta_+ Q)_{p-1},t_{p-1},h \big) = \Delta_- \big( f (\Q,\T,h) \big)_p +g(t_p,h).
\end{equation}
This concludes the proof.
\end{proof}

\bibliographystyle{plain}

\end{document}